\documentclass[12pt]{amsart}

\usepackage[margin=1in]{geometry}

\usepackage{amsmath,amssymb,amsthm, enumerate, url}
\usepackage{xcolor}

\newtheorem{theorem}{Theorem}
\newtheorem{lemma}[theorem]{Lemma}
\newtheorem{proposition}[theorem]{Proposition}

\theoremstyle{remark}
\renewcommand\O{\mathcal{O}}
\newcommand\calS{\mathcal{S}}
\newcommand\Z{\mathbb{Z}}
\newcommand\htt{\mathrm{ht}}

\numberwithin{equation}{section}
\numberwithin{theorem}{section}

\title{Upper bounds on polynomials with small Galois group}

\author{Robert J. Lemke Oliver}
\address{Department of Mathematics, Tufts University, 503 Boston Ave, Medford, MA 02155}
\email{robert.lemke\_oliver@tufts.edu}

\author{Frank Thorne}
\address{Department of Mathematics, University of South Carolina, 1523 Greene St, Columbia, SC 29201}
\email{thorne@math.sc.edu}

\begin{document}

\maketitle

\begin{abstract}
When monic integral polynomials of degree $n \geq 2$ are ordered by the maximum of the absolute value of their coefficients, the Hilbert irreducibility theorem implies that asymptotically $100\%$ are irreducible and have Galois group isomorphic to $S_n$.  In particular, amongst such polynomials whose coefficients are bounded by $B$ in absolute value, asymptotically $(1+o(1))(2B+1)^n$ are irreducible and have Galois group $S_n$.  When $G$ is a proper transitive subgroup of $S_n$, however, the asymptotic count of polynomials with Galois group $G$ has been determined only in very few cases.  

Here, we show that if there are strong upper bounds on the number of degree $n$ fields with Galois group $G$, then there are also strong bounds on the number of polynomials with Galois group $G$.  For example, for any prime $p$, we show that there are at most $O(B^{3 - \frac{2}{p}} (\log B)^{p - 1})$ polynomials with Galois group $C_p$ and coefficients bounded by $B$.
\end{abstract}

\section{Introduction}

Fix an integer $n \geq 2$ and let $G$ be a subgroup of $S_n$.  For any $B \geq 1$, let
\[
P_n(B;G) 
	:= \#\{ f  \in \mathbb{Z}[x] \text{ monic, degree $n$}: \mathrm{ht}(f) \leq B, \mathrm{Gal}(f) \simeq G\},
\]
where $\mathrm{ht}(f)$ is the maximum absolute value of the coefficients of $f$. 
The Hilbert irreducibility theorem implies that asymptotically $100\%$ of such polynomials are irreducible and have Galois group $S_n$, so that $P_n(B;S_n) \sim (2B+1)^n$.  There are very few other groups, however, for which the asymptotic order of $P_n(B;G)$ is known, even in small degrees; see \cite{ChowDietmann} for recent results in degrees $3$ and $4$.  Instead, one often must settle for placing upper bounds on $P_n(B;G)$.

There are two broad approaches in the literature to do so.  The first uses the large sieve (often in conjunction with Serre's notion of thin sets) to bound from above the number of polynomials that avoid certain behavior modulo small primes; the record via this approach is due to 
Gallagher \cite{Gallagher}, who shows that $P_n(B; G) \ll B^{n - \frac12} (\log B)^{1 - \gamma_n}$ with $\gamma_n > 0$ for all $G \not \simeq S_n$ simultaneously, and 
Zywina \cite{Zywina}, who improves this to $P_n(B;G) \ll B^{n - \frac{1}{2}}$ for large $n$.

The other approach is to bound the number of polynomials such that some auxiliary Galois resolvent has an exceptional property; the pioneer  of this approach and the overall record-holder is Dietmann \cite{DietmannDist, Dietmann}, who shows that $P_n(B;G) \ll B^{n - 1 + \frac{1}{[S_n:G]} + \epsilon}$ for $G \neq A_n$ and that
 $P_n(B;A_n) \ll B^{n - 2 + \sqrt{2} + \epsilon}$.
 
Here, we introduce a third approach that has the potential to yield substantially stronger results but 
requires a hypothesis that is unproved in many cases. 
Before stating a general theorem, we begin with two unconditional results that are reflective of our method.

\begin{theorem}\label{thm:cyclic}
Let $p \geq 3$ be prime and let $C_p$ denote the cyclic group of order $p$.  Then $P_p(B;C_p) \ll B^{3 - \frac{2}{p}} (\log B)^{p - 1}$.
\end{theorem}

More generally, for any $n \geq 2$, let $P_n^\mathrm{gal}(B)$ denote the number of monic, irreducible polynomials $f \in \mathbb{Z}[x]$ whose coefficients are bounded in absolute value by $B$, and for which the field $\mathbb{Q}(x)/f(x)$ is Galois.

\begin{theorem}\label{thm:galois}
If $n \geq 5$, then $P_n^\mathrm{gal}(B) \ll B^{\frac{3}{4}n + \frac{1}{4} + \epsilon}$.
\end{theorem}

Our approach is as follows.  When $G \subseteq S_n$ is transitive, any polynomial $f(x)$ counted by $P_n(B;G)$ cuts out a field $\mathbb{Q}(x)/f(x)$ whose normal closure has Galois group $G$ and whose discriminant is $O_n(B^{2n-2})$. In Section \ref{sec:multiplicity}, we bound from above the multiplicity with which a given field is cut out in this manner.  It follows that if there are not too many fields with Galois group $G$, then there can also not be too many polynomials with Galois group $G$.  To make this precise, let
\[
\mathcal{F}_n(X;G)
	:= \{ K/\mathbb{Q} : [K:\mathbb{Q}] = n, \mathrm{Gal}(\widetilde{K}/\mathbb{Q}) \simeq G, |\mathrm{Disc}(K)| \leq X\},
\]
where $\widetilde{K}$ is the normal closure of $K/\mathbb{Q}$ and $\mathrm{Disc}(K)$ is the absolute discriminant of $K$, and set $N_n(X;G) := \#\mathcal{F}_n(X;G)$.

\begin{theorem}\label{thm:general}
With notation as above, assume for some constant $e > \frac{1}{n^2-n}$ that the bound $N_n(X;G) \ll X^e$ holds for every sufficiently large $X$, with an implied constant depending at most on $G$ and $e$.  Then
\[
P_n(B;G) 
	\ll B^{e(2n-2) + 1}(\log B)^{n-1}.
\]
If the group $G$ is {primitive}, then this may be improved to $P_n(B;G) \ll B^{e(2n-2) + 1 - \frac{2}{n}}(\log B)^{n-1}$.
\end{theorem}

We recall that a permutation group $G$ on $n$ elements is said to be primitive if it preserves no nontrivial partition of the elements.  For example, when $p$ is prime, this hypothesis is satisfied for every transitive subgroup of $S_p$, so that Theorem \ref{thm:cyclic} follows from Theorem \ref{thm:general} and the upper bound $N_p(X;C_p) \ll X^{\frac{1}{p-1}}$ \cite{maki, wright}.  Theorem \ref{thm:galois}, on the other hand, follows from the first case of Theorem \ref{thm:general} and an upper bound of $O(X^{3/8 + \epsilon})$ due to Ellenberg and Venkatesh \cite[Proposition 1.3]{EllenbergVenkatesh} on the number of Galois number fields.  This result of Ellenberg and Venkatesh is not sharp, and correspondingly neither is Theorem \ref{thm:galois}.  It could likely be improved substantially for any given $n$ with modest effort.

Theorem \ref{thm:general} improves on the results provided by studying thin sets and Galois resolvents desribed above if for example its hypothesis holds with some $e \leq \frac{1}{2} - \frac{1}{2n-2}$; slightly weaker results suffice for any particular group $G$.
Malle's conjecture \cite{Malle} predicts an asymptotic of the form $N_n(X;G) \sim c(G) X^{a(G)} (\log X)^{b(G)}$, where the constant 
$a(G)$ is the inverse of an integer and satisfies $\frac{1}{n - 1} \leq a(G) \leq 1$. 
Thus, for $n \geq 5$, Theorem \ref{thm:general} conjecturally improves upon prior work whenever $a(G) \leq \frac13$.   This criterion is satisfied precisely when $G \subseteq S_n$ is such that for every $1 \neq g \in G$, the permutation action of $g$ decomposes into at most $n - 3$ orbits.

\section*{Acknowledgements}
The authors would like to thank Sam Chow, Rainer Dietmann, Michael Filaseta, and Martin Widmer for helpful feedback.

RJLO was partially supported by NSF grant DMS-1601398.  FT was partially supported by grants from the Simons Foundation (Nos. 563234 and 586594).

\section{Multiplicities of polynomials cutting out fields}
\label{sec:multiplicity}

Given a number field $K$ of degree $n$ and signature $(r_1,r_2)$, let
\[
M_K(B)
	:= \#\{ f \in \mathbb{Z}[x] \text{ monic, degree $n$} : \mathrm{ht}(f) \leq B, \mathbb{Q}(x)/f \simeq K\}
\]
be the multiplicity with which $K$ is cut out amongst polynomials of height at most $B$.

 The main result of this section is the following upper bound on $M_K(B)$.
\begin{theorem}\label{thm:multiplicity}
Write
\[
\lambda = \min\{ ||\alpha|| : \alpha \in \mathcal{O}_K \setminus \mathbb{Z}\},
\]
where for each $\alpha \in \O_K$ we write $||\alpha||$ for the maximum absolute value of $\alpha$ under the embeddings of $K$.

Then, we have
\[
M_K(B) \ll \frac{B (\log B)^{r_1+r_2-1}}{\lambda}.
\]
\end{theorem}

Loosely speaking, the idea of the proof of Theorem \ref{thm:multiplicity} is that most elements $\alpha \in \mathcal{O}_K$ should have minimal polynomials $f_\alpha(x) = x^n + a_1 x^{n-1} + \dots + a_n$ with coefficients scaling like $|a_i| \approx ||\alpha||^i$ rather than having all coefficients of roughly the same size.  We show that the elements whose minimal polynomials are counted by $M_K(B)$ lie in a very restricted region inside Minkowski space, and then bound from above the number of elements inside that region.

To make this precise, we begin by recalling the definition of the \emph{Mahler measure} of a polynomial.  Given a monic polynomial $f \in \mathbb{C}[x]$, its Mahler measure $m(f)$ is defined to be
\[
m(f) := \prod_{\theta: f(\theta)=0} \max\{1,|\theta|\},
\]
where the product runs over the roots of $f$ with multiplicity.  A standard argument using 
Jensen's formula, which we learned about from \cite{zerocollar},
establishes that if we write $f(x) = x^n + a_1 x^{n-1} + \dots + a_n$, then
\begin{equation}\label{eq:jensen}
m(f) \leq \sqrt{1 + a_1^2 + \dots + a_n^2} \leq \sqrt{n + 1} \cdot \htt(f).
\end{equation}

Let $K \hookrightarrow \mathbb{R}^{r_1} \times \mathbb{C}^{r_2} =: K_\infty$ be the standard embedding of $K$ into Minkowski space.  For any $x \in K_\infty$, we may analogously define its Mahler measure $m(x)$ by
\[
m(x) := \prod_{\sigma} \max\{1, |x_\sigma|^{\mathrm{deg}\,\sigma}\}.
\]
For any $Y \geq 1$, we let $\Omega_Y \subseteq K_\infty$ consist of those elements whose Mahler measure is at most $Y$. 
Then by \eqref{eq:jensen} we have
\begin{equation}\label{eq:mk_bound}
M_K(B) \leq \#\{\Omega_{\sqrt{n + 1} \cdot B} \cap (\mathcal{O}_K - \Z)\},
\end{equation}
allowing us to proceed by bounding the right-hand side of \eqref{eq:mk_bound}. 
The volume of $\Omega_Y$, or of its projection onto any subset of the coordinate axes, is easily computed to be
$O_n(Y (\log Y)^{r_1+r_2-1})$. It therefore follows from {\itshape Davenport's lemma} \cite{dav_lemma} that the simpler count $\#\{ x \in \Omega_Y \cap \mathbb{Z}^n \}$ satisfies
\begin{equation}\label{eqn:davenport}
\# \{ x \in \Omega_Y \cap \mathbb{Z}^n \} \ll Y (\log Y)^{r_1+r_2-1}.
\end{equation}
For the sake of motivation, we begin by proving a weaker result than Theorem \ref{thm:multiplicity} that recovers the bound in \eqref{eqn:davenport} but for $\#\{\Omega_{\sqrt{n + 1} \cdot B} \cap (\mathcal{O}_K - \Z)\}$.

\begin{proposition}\label{lem:first-multiplicity}
With notation as above, $M_K(B) \ll B (\log B)^{r_1+r_2-1}$.
\end{proposition} 

\begin{proof}
Let $\mu \asymp_n 1$ be the shortest length of any nonzero vector in $\O_K$.
Let $\eta > 0$ be an arbitrary small constant, write $\delta := \frac{\mu}{\sqrt{n}} - \eta$, 
and consider the dilation $\Lambda := \delta \Z^n$ of the integer lattice $\Z^n \subseteq K_\infty$. By construction, the box of side length $\delta$
centered at any point in $\delta \Z^n$ contains at most one point of $\O_K$. To each $\alpha \in \O_K$ we associate the nearest vector 
$v_\alpha \in \delta \Z^n$, choosing arbitrarily if there is more than one such. Then the map $\alpha \mapsto v_\alpha$ is injective and we write $\calS := \{ v_\alpha : \alpha \in \O_K \}$.

Let $Y = \sqrt{n+1} \cdot B$, so that $M_K(B) \leq \#(\Omega_Y \cap \mathcal{O}_K)$.
If $\alpha \in \Omega_Y \cap \O_K$, we have
$\frac{1}{\delta} v_\alpha - x \in \frac{1}{\delta} \Omega_Y$ for some
$x \in \textnormal{Box}(1)$, the unit box centered at the origin.
Since $\delta \asymp 1$, we have
\begin{align*}\label{eq:use_dav2}
\#\{\Omega_{Y} \cap \mathcal{O}_K\}
& \leq
\# \Big\{ \frac{1}{\delta} \calS \cap \Big(\frac{1}{\delta} \Omega_{Y} + \textnormal{Box}(1) \Big) \Big\} \\
& \leq
\# \Big\{ \Z^n \cap \Big(\frac{1}{\delta} \Omega_{Y} + \textnormal{Box}(1) \Big) \Big\} \\
& \ll
Y \log^{r_1 + r_2 - 1}(Y),
\end{align*}
as desired, with the last inequality following from Davenport's lemma.
\end{proof}

\begin{proof}[Proof of Theorem \ref{thm:multiplicity}]
In the proof of Proposition \ref{lem:first-multiplicity}, we still let $Y = \sqrt{n+1}\cdot B$ but now choose $\delta := \frac{\lambda}{\sqrt{n}} - \eta$ and define
$\alpha$ and $\calS$ as before. The map $\alpha \mapsto v_\alpha$ is no longer injective, but it has the property that 
 $\alpha - \alpha' \in \Z$ whenever $v_\alpha = v_{\alpha'}$, and hence it has preimages of size $O(\lambda)$.

We now decompose the region $\Omega_Y$ as $\Omega_Y = \bigcup_{s_1, s_2} \Omega_{Y, s_1, s_2}$, where 
\[
\Omega_{Y, s_1, s_2} := \{ x \in \Omega_Y \ : \ |x_\sigma| \geq \delta \textnormal{ for exactly } s_1 \textnormal{ real and } s_2 \textnormal{ complex places } \delta \}.
\]
For each $\Omega_{Y, s_1, s_2}$, we now have that
\begin{align*}\label{eq:use_dav3}
\#\{\Omega_{Y, s_1, s_2} \cap \mathcal{O}_K\}
& \leq
O(\lambda) \cdot \# \Big\{ \frac{1}{\delta} \calS \cap \Big(\frac{1}{\delta} \Omega_{Y, s_1, s_2} + \textnormal{Box}(1) \Big) \Big\} \\
& \leq
O(\lambda) \cdot \# \Big\{ \Z^n \cap \Big(\frac{1}{\delta} \Omega_{Y, s_1, s_2} + \textnormal{Box}(1) \Big) \Big\} \\
& \ll
\lambda \cdot \frac{Y (\log Y)^{r_1 + r_2 - 1}}{\delta^{s_1 + 2s_2}} \\
& \ll
\frac{Y (\log Y)^{r_1 + r_2 - 1}}{\lambda^{s_1 + 2s_2 - 1}}.
\end{align*}
This bound is as desired when $s_1 + 2s_2 \geq 2$, leaving only the exceptional regions $\Omega_{Y,0,0}$ and $\Omega_{Y,1,0}$,  the latter region being nonempty only if $K$ has a real place.

The region $\Omega_{Y, 0, 0}$, if it is nonempty, is contained within a single box of length $\delta$ centered at the origin, and hence contains
only rational integers that do not contribute to $M_K(B)$.

To bound $\#(\Omega_{Y,1, 0} \cap \mathcal{O}_K)$, note that for each $\alpha \in \Omega_{Y, 1, 0}$, 
the corresponding $v_\alpha$ has exactly one nonzero coordinate.
For each such $v$, the number of $\alpha \in \mathcal{O}_K \cap \Omega_{Y, 1,0}$ with $v_\alpha = v$ is bounded above by the number of 
integers $k$ with $m(v + k) \leq Y$, which is
$\ll \big( \frac{Y}{||v||} \big)^{\frac{1}{n - 1}}$.
We therefore have
\[
\#(\Omega_{Y,1, 0} \cap \mathcal{O}_K)
	\ll
	 \sum_{1 \leq j \leq \frac{Y}{\delta} + 1} \left(\frac{Y}{\delta j}\right)^{\frac{1}{n-1}} \ll \frac{Y \log Y}{\delta} \asymp \frac{Y \log Y}{\lambda},
\]
the factor of $\log Y$ being extraneous unless $n = 2$, but harmless to include in all cases.
 
The theorem now follows by summing the above bounds over all $s_1$ and $s_2$.
\end{proof}

\section{Proof of Theorem \ref{thm:general}}

We now turn to the proof of Theorem \ref{thm:general}.  If $f \in \mathbb{Z}[x]$ is monic of degree $n$ and satisfies $\mathrm{ht}(f) \leq B$, then its discriminant satisfies $|\mathrm{Disc}(f)| \ll B^{2n-2}$.  It follows that there is some positive constant $c$ such that
\[
P_n(B;G)
	\leq \sum_{K \in \mathcal{F}_n(c B^{2n-2}; G)} M_K(B).
\]
Appealing to Proposition \ref{lem:first-multiplicity}, we see that $M_K(B) \ll B (\log B)^{r_1+r_2-1} \leq B (\log B)^{n-1}$.  By our hypothesis that $N_n(X;G) \ll X^e$ for every sufficiently large $X$, we conclude that
\[
P_n(B;G)
	\ll B^{e(2n-2) + 1} (\log B)^{n-1},
\]
yielding the first case of the theorem.

For the second case, we require the following lemma of Ellenberg and Venkatesh.

\begin{lemma}\label{lem:height-bound}
Let $K/\mathbb{Q}$ be an extension of degree $n$.  Then for any $\alpha \in \mathcal{O}_K$ such that $K=\mathbb{Q}(\alpha)$, there holds $||\alpha|| \gg |\mathrm{Disc}(K)|^{\frac{1}{n(n-1)}}$.
\end{lemma}
\begin{proof}
This is essentially \cite[Lemma 3.1]{EllenbergVenkatesh}, but we recap the proof here because it is brief.  The minimal polynomial $f_\alpha$ of such an $\alpha$ has discriminant bounded above by $O(||\alpha||^{n(n-1)})$.  Since the discriminant of $f_\alpha$ is an upper bound on $\mathrm{Disc}(K)$, the result follows.
\end{proof}

If the group $G$ is primitive, then any field $K \in \mathcal{F}_n(X;G)$ has no proper subextensions.  By Lemma \ref{lem:height-bound}, it follows that for such fields, we may take $\lambda = |\mathrm{Disc}(K)|^{\frac{1}{n(n-1)}}$ in Theorem \ref{thm:multiplicity}.  So doing, we have
\[
P_n(B;G)
	\ll \sum_{K \in \mathcal{F}_n(c B^{2n-2}; G)} \frac{B (\log B)^{n-1}}{ |\mathrm{Disc}(K)|^{\frac{1}{n(n-1)}}}
	\ll B^{e(2n-2)+1-\frac{2}{n}} (\log B)^{n-1}
\]
by partial summation and our assumption that $e > \frac{1}{n^2-n}$.  This completes the proof.

\section{Further examples}

While obtaining upper bounds on $N_n(X;G)$ that are useful in Theorem \ref{thm:general} is difficult in general, there are some groups $G$ for which strong results are available.  
For example Kl\"uners and Malle \cite{KlunersMalle} proved that for a nilpotent group 
$G \subseteq S_n$ of order $n$ acting in its regular representation, we have $N_n(X; G) \ll X^e$ with $e = \frac{\ell}{n(\ell - 1)} + \epsilon$,
where $\ell$ is the smallest prime divisor of $n$. We therefore obtain:

\begin{theorem}\label{thm:nilpotent}
Let $G$ be a nilpotent group of order $n$, and let $\ell$ be the smallest prime divisor of $|G|$.  Then for any $\epsilon>0$,
\[
P_n(B;G)
	\ll_\epsilon B^{\frac{2\ell}{\ell -1} + 1 - \frac{2\ell}{(\ell-1)n} + \epsilon}.
\]
\end{theorem}

We recall for the convenience of the reader that abelian groups are nilpotent, so Theorem \ref{thm:nilpotent} applies whenever $G$ is abelian. Indeed,
in the abelian case, Wright's earlier work \cite{wright} implies that the factor of $B^{\epsilon}$ may be replaced by a power of $\log B$.

In fact, Malle's conjecture \cite{Malle} implies that the conclusion of Theorem \ref{thm:nilpotent} should hold for any group $G$ of order $n$, not just nilpotent groups; this suggests for example that $P_n^\mathrm{gal}(B) \ll B^{5 - \frac{4}{n}+\epsilon}$ for every $n$.  Additionally, recent work of Alberts \cite{Alberts} establishes the weak Malle conjecture for nilpotent groups $G$ in every representation, not just their regular representation; in many cases, this would yield strong upper bounds on $P_n(B;G)$ for nilpotent groups $G$ with degree $n$ representations.

\bibliographystyle{alpha}
\bibliography{refs-polynomialgalois}

\begin{thebibliography}{{Zyw}10}

\bibitem[{Alb}18]{Alberts}
Brandon {Alberts}.
\newblock {The Weak Form of Malle's Conjecture and Solvable Groups}.
\newblock {\em Preprint}, 2018.
\newblock Available at \url{https://arxiv.org/abs/1804.11318}.

\bibitem[CD18]{ChowDietmann}
Sam {Chow} and Rainer {Dietmann}.
\newblock {Enumerative Galois theory for cubics and quartics}.
\newblock {\em Preprint}, 2018.
\newblock Available at \url{https://arxiv.org/abs/1807.05820}.

\bibitem[Dav51]{dav_lemma}
H.~Davenport.
\newblock On a principle of {L}ipschitz.
\newblock {\em J. London Math. Soc.}, 26:179--183, 1951.

\bibitem[Die12]{DietmannDist}
Rainer Dietmann.
\newblock On the distribution of {G}alois groups.
\newblock {\em Mathematika}, 58(1):35--44, 2012.

\bibitem[Die13]{Dietmann}
Rainer Dietmann.
\newblock Probabilistic {G}alois theory.
\newblock {\em Bull. Lond. Math. Soc.}, 45(3):453--462, 2013.

\bibitem[EV06]{EllenbergVenkatesh}
Jordan~S. Ellenberg and Akshay Venkatesh.
\newblock The number of extensions of a number field with fixed degree and
  bounded discriminant.
\newblock {\em Ann. of Math. (2)}, 163(2):723--741, 2006.

\bibitem[Gal73]{Gallagher}
P.~X. Gallagher.
\newblock The large sieve and probabilistic {G}alois theory.
\newblock In {\em Analytic number theory ({P}roc. {S}ympos. {P}ure {M}ath.,
  {V}ol. {XXIV}, {S}t. {L}ouis {U}niv., {S}t. {L}ouis, {M}o., 1972)}, pages
  91--101. Amer. Math. Soc., Providence, R.I., 1973.

\bibitem[KM04]{KlunersMalle}
J\"{u}rgen Kl\"{u}ners and Gunter Malle.
\newblock Counting nilpotent {G}alois extensions.
\newblock {\em J. Reine Angew. Math.}, 572:1--26, 2004.

\bibitem[M{\"a}k85]{maki}
Sirpa M{\"a}ki.
\newblock On the density of abelian number fields.
\newblock {\em Ann. Acad. Sci. Fenn. Ser. A I Math. Dissertationes}, (54):104,
  1985.

\bibitem[Mal04]{Malle}
Gunter Malle.
\newblock On the distribution of {G}alois groups. {II}.
\newblock {\em Experiment. Math.}, 13(2):129--135, 2004.

\bibitem[Ram15]{zerocollar}
Ramya.
\newblock Upper bound for {M}ahler's measure.
\newblock {\em Blog post}, 2015.
\newblock Available at
  \url{http://zerocollar.blogspot.com/2015/04/upper-bound-for-mahlers-measure.html}.

\bibitem[Wri89]{wright}
David~J. Wright.
\newblock Distribution of discriminants of abelian extensions.
\newblock {\em Proc. London Math. Soc. (3)}, 58(1):17--50, 1989.

\bibitem[{Zyw}10]{Zywina}
David {Zywina}.
\newblock {Hilbert's irreducibility theorem and the larger sieve}.
\newblock {\em Preprint}, 2010.
\newblock Available at \url{https://arxiv.org/abs/1011.6465}.

\end{thebibliography}

\end{document}